\documentclass{amsart}
\usepackage{amsmath}
\usepackage{amssymb}
\usepackage[all]{xy}
\usepackage{enumerate}
\numberwithin{equation}{section}

\theoremstyle{plain}
\newtheorem{thm}{Theorem}[section]
\newtheorem{lemma}[thm]{Lemma}

\newtheorem{prop}[thm]{Proposition}

\theoremstyle{definition}

\newtheorem{remark}[thm]{Remark}

\begin{document}

\title{Support Varieties For Frobenius Kernels of Classical Groups}

\author[Paul Sobaje]
{Paul Sobaje}

\begin{abstract}

\sloppy{
Let $G$ be a classical simple algebraic group over an algebraically closed field $k$ of characteristic $p>0$, and denote by $G_{(r)}$ the $r$-th Frobenius kernel of $G$.  We show that for $p$ large enough, the support variety of a simple $G$-module over $G_{(r)}$ can be described in terms of support varieties of simple $G$-modules over $G_{(1)}$.  We use this, together with the computation of the varieties $V_{G_{(1)}}(H^0(\lambda))$, given by Jantzen in \cite{J2} and by Nakano et al. in \cite{NPV}, to explicitly compute the support variety of a block of $Dist(G_{(r)})$.
}

\end{abstract}

\maketitle

The aim of this paper is to provide computations of support varieties for modules over Frobenius kernels of algebraic groups.  Specifically, for $G$ a classical simple algebraic group over an algebraically closed field $k$ of characteristic $p>0$, we give a description (Theorem \ref{simple}) of the support variety of a simple $G$-module over the $r$-th Frobenius kernel $G_{(r)}$ in terms of the support varieties of simple $G$-modules over $G_{(1)}$.  Our proofs establish these results only under the assumption that $p$ is large enough for the root system of $G$.  A lower bound on $p$ is provided in Section 3, roughly speaking it is the Coxeter number of $G$ multiplied by a quadratic polynomial in the rank of $G$.  In Section 4, we apply this result for $G = SL_n$ or $Sp_{2n}$, to give an explicit description of the support variety of a block of the distribution algebra $Dist(G_{(r)})$.

We should emphasize that the varieties computed in Section 3 can only be determined explicitly (by our results) if the support varieties of simple $G$-modules over $G_{(1)}$ are known explicitly, which is in general not the case.  However, C. Drupieski, D. Nakano, and B. Parshall have in recent work \cite{DNP} made such calculations for simple, simply-connected $G$, if one assumes that $p$ is at least as large as the Coxeter number of $G$ and that Lusztig's character formula holds for all restricted dominant weights.

The results in this paper rely most heavily on the work of A. Suslin, E. Friedlander, and C. Bendel in \cite{SFB1} and \cite{SFB2}.  In particular, all of our statements of support varieties are given in terms of varieties of 1-parameter subgroups, which the aforementioned papers prove to be homeomorphic to cohomologically defined support varieties.  Moreover, the intuition behind our results for simple modules came from the calculations made in \cite[6.10]{SFB2} for Frobenius kernels of $SL_2$.  We also use in an essential way the analysis and results of J. Carlson, Z. Lin, and D. Nakano in \cite{CLN}, and that of E. Friedlander in \cite{F}, both of which appear in the proof of Proposition \ref{first}.  Finally, the results of J. Jantzen in \cite{J2}, and the results and observations of D. Nakano, B. Parshall, and D. Vella in \cite{NPV} are critical to obtaining the calculations found in Section 4, where we compute the support variety of a block of $Dist(G_{(r)})$.

\section{Preliminaries}

We will assume throughout that $k$ is an algebraically closed field of characteristic $p>0$.

\subsection{Representations of $G$} By a ``classical" simple algebraic group, we shall mean that $G$ is one of the groups $SL_n, SO_n$, or $Sp_{2n}$ (thus excluding the simply-connected groups of types $B$ and $D$).  When viewing $G$ as a subgroup of some $GL_n$, we will always assume this embedding is the ``natural" one associated to $G$.

Let $T$ be a maximal split torus of $G$ with character group $X(T)$, let $\Phi$ be the root system for $G$ with respect to $T$, and fix a set of simple roots $\Pi = \{\alpha_1,\ldots,\alpha_{\ell}\}$.  Denote by $\Phi^+$ the set of positive roots with respect to $\Pi$, and let $B^+$ and $B$ denote the Borel subgroups corresponding to $\Phi^+$ and $-\Phi^+$, with their unipotent radicals denoted as $U^+, U$ respectively.  The Weyl group of $\Phi$ will be denoted $W$, and the dot action of $w \in W$ on $\lambda \in X(T)$ is defined by $w \cdot \lambda = w(\lambda + \rho) - \rho$, where $\rho$ is the half sum of the positive roots.  We also denote by $\alpha_0$ the highest root.

We let $\alpha^{\vee} = 2\alpha/\langle \alpha, \alpha \rangle$ for all roots $\alpha$.  The dominant integral weights of $X(T)$ are then given by

\vspace{0.05 in}
\begin{center}$X(T)_+ := \{ \lambda \in X(T) \mid 0 \leq \langle \lambda, \alpha_i^{\vee} \rangle, 1 \leq i \leq \ell \}.$\end{center}
\vspace{0.05 in}

The set of fundamental dominant weights, $\{ \omega_1, \ldots, \omega_{\ell} \}$, is defined by $\langle \omega_i, \alpha_j^{\vee} \rangle = \delta_{ij}$.  For each $\lambda \in X(T)_+$, we denote by $L(\lambda)$ the unique simple $G$-module of highest weight $\lambda$.  It is the socle of the induced module $\text{H}^0(\lambda) := Ind_B^G(k_{\lambda})$, where $k_{\lambda}$ is the simple one-dimensional $B$-module of weight $\lambda$.  The morphism $F: G \rightarrow G$ is the standard Frobenius morphism on $GL_n$ restricted to $G$, and $G_{(r)} \subseteq G$ is the kernel of $F^r$.  For a $G$-module $M$, we denote by $M^{(r)}$ the module which arises from pulling back $M$ via $F^r$.  The set of $p^r$-restricted weights of $X(T)$ is given by

\vspace{0.05 in}
\begin{center}$X_r(T) := \{ \lambda \in X(T) \mid 0 \leq \langle \lambda, \alpha_i^{\vee} \rangle < p^r, 1 \leq i \leq \ell \}.$\end{center}
\vspace{0.05 in}

As shown in \cite[II.3]{J1}, if $\lambda \in X_r(T)$, then $L(\lambda)$ remains simple upon restriction from $G$ to $G_{(r)}$.  Moreover, for $G$ simply-connected, the set 

\vspace{0.05 in}
\begin{center}$\{ L(\lambda) \mid \lambda \in X_r(T) \}$,\end{center}
\vspace{0.05 in}

\noindent is a complete set of pairwise non-isomorphic simple $G_{(r)}$-modules.

\vspace{0.05 in}
\subsection{Distribution Algebras}
If $H$ is any affine group scheme, with coordinate algebra $k[H]$, and $I_{\epsilon}$ the augmentation ideal of $k[H]$, then the distribution algebra of $H$, $Dist(H)$, is defined by

\vspace{0.05 in}
\begin{center}$Dist(H) = \{ f \in \text{Hom}_k(k[H],k) \mid f(I_{\epsilon}^n) = 0, \text{ for some } n \ge 1\}$.\end{center}
\vspace{0.05 in}

\noindent It follows that $Dist(H_{(r)}) \subseteq Dist(H_{(r+1)})$, and $Dist(H) = \bigcup_{r \ge 1} Dist(H_{(r)})$ (see \cite[I.9]{J1} for more on Frobenius kernels of arbitrary affine group schemes).  For a morphism of affine group schemes $\phi: H_1 \rightarrow H_2$, we denote by $d\phi: Dist(H_1) \rightarrow Dist(H_2)$ the induced map of algebras.

Of particular importance will be the structure of the algebra $Dist(\mathbb{G}_a)$.  In this case, we have $k[\mathbb{G}_a] \cong k[t]$, and $Dist(\mathbb{G}_a)$ is spanned by the elements $(\frac{d}{dt})^{(j)}$, where

$$\left(\frac{d}{dt}\right)^{(j)}(t^i) = \delta_{ij}$$

If we set $u_j = (\frac{d}{dt})^{(p^j)}$, and if $m$ is an integer with $p$-adic expansion $m = m_0 + m_1p + \cdots + m_qp^q$, then

$$\left(\frac{d}{dt}\right)^{(m)} = \frac{u_0^{m_0}\cdots u_q^{m_q}}{m_0!\cdots m_q!}$$

Therefore $Dist(\mathbb{G}_a)$ is generated as an algebra over $k$ by the set $\{u_j\}_{j \ge 0}$, while $Dist(\mathbb{G}_{a(r)})$ is generated by the subset where $j<r$.

With $F^i$ denoting the $i$-th iterate of the Frobenius morphism as above, we have that the differential $dF^i: Dist(\mathbb{G}_a) \rightarrow Dist(\mathbb{G}_a)$ is given by

$$dF^i(u_j) = \begin{cases} u_{j-i} & \text{ if } j \ge i \\ 0 & \text{ otherwise}\\ \end{cases} $$

Let $\delta: \mathbb{G}_a \rightarrow \mathbb{G}_a \times \mathbb{G}_a$ be the morphism which sends $g$ to $(g,g)$, for all $g \in \mathbb{G}_a(A)$, and for all commutative $k$-algebras $A$.  Then the differential of $\delta$ is the co-multiplication of $Dist(\mathbb{G}_a)$ (see \cite[I.7.4]{J1}), so we will write $d\delta$ as $\Delta_{\mathbb{G}_a}^{\prime}$.  It is not hard to show that  $\Delta_{\mathbb{G}_a}^{\prime}: Dist(\mathbb{G}_a) \rightarrow Dist(\mathbb{G}_a) \otimes Dist(\mathbb{G}_a)$, is given by 

$$\Delta_{\mathbb{G}_a}^{\prime} \left( \left( \frac{d}{dt} \right)^{(n)} \right) = \sum_{i+j = n} \left( \frac{d}{dt} \right)^{(i)} \otimes \left( \frac{d}{dt} \right)^{(j)}.$$

\bigskip
\subsection{Support Varieties}
We recall that the Frobenius kernel $H_{(r)}$ has finite dimensional coordinate algebra $k[H_{(r)}]$, and thus is a \textit{finite group scheme}.  By \cite{FS}, we have then that the algebra

$$\text{H}^{\bullet}(H_{(r)},k) := \begin{cases} \bigoplus_{i \ge 0} \text{H}^i(H_{(r)},k) & \text{if char $k = 2$} \\  \bigoplus_{i \ge 0} \text{H}^{2i}(H_{(r)},k) & \text{if char $k \ne 2$} \end{cases}$$

\noindent is a finitely generated commutative algebra over $k$.  Denote by $|H_{(r)}|$ the maximal ideal spectrum of $\text{H}^{\bullet}(H_{(r)},k)$.  For $M$ a finitely generated $H_{(r)}$-module, the cohomological support variety of $M$, denoted $|H_{(r)}|_M$, is defined to be set of all $\mathfrak{m} \in |H_{(r)}|$ that contain the annihilator in $\text{H}^{\bullet}(H_{(r)},k)$ of $\text{Ext}_{H_{(r)}}^{\bullet}(M,M)$, where the latter space is considered as a module over the former via the cup product.

In this paper we will study these varieties via the non-cohomological description provided in \cite{SFB2}.  Denote by $V_{H_{(r)}}$ the set $\text{Hom}_{GrpSch/k}(\mathbb{G}_{a(r)}, H_{(r)})$, which is the variety of 1-parameter subgroups of $H_{(r)}$.  For $M \in H_{(r)}$-mod, let $V_{H_{(r)}}(M)$ be the set of those $\sigma \in V_{H_{(r)}}$ having the property that $M$ is not projective as a $k[u]/(u^p)$-module, where the action of $u$ is given by $d\sigma(u_{r-1})$.  The following theorem is proved in \cite{SFB2}, although our formulation is more easily seen in \cite[Theorem 1.9]{FP} (and the discussion preceding it).

\begin{thm}\label{alternate}
There is a natural homeomorphism $\Psi: V_{H_{(r)}}  \stackrel{\sim}{\rightarrow} |H_{(r)}|$ such that $V_{H_{(r)}}(M) = \Psi^{-1}(|H_{(r)}|_M)$.
\end{thm}

Since we are assuming that $G$ is a classical simple algebraic group, there is in this case a very nice description of the variety $V_{G_{(r)}}$.  As shown in \cite[Remark 1.3]{SFB1}, for any $x \in \mathfrak{gl}_n$ with $x^p = 0$, there is a morphism of algebraic groups $\text{exp}_x: \mathbb{G}_a \rightarrow GL_n$, where for all commutative $k$-algebras $A$, and all $s \in \mathbb{G}_a(A)$, we have

$$\text{exp}_x(s) = exp(sx) = 1 + sx + \frac{(sx)^2}{2} + \cdots + \frac{(sx)^{p-1}}{(p-1)!}$$

Let $\mathfrak{g} = Lie(G)$.  As $G$ is a subgroup of $GL_n$ via the representation of its natural module, $\mathfrak{g} \subseteq \mathfrak{gl}_n$.  Moreover, this embedding of $G$ is one of \textit{exponential type}, meaning that if $x \in \mathfrak{g}$ is $p$-nilpotent, then the map $\text{exp}_x$ defined above factors through $G$ (see \cite[Lemma 1.8]{SFB1}).  Now let $\mathcal{N}_p(\mathfrak{g})$ denote the set of those $x  \in \mathfrak{g}$ such that $x^{[p]} = 0$, and let

\vspace{0.05 in}
\begin{center}$C_r(\mathcal{N}_p(\mathfrak{g})) := \{ (\beta_0, \beta_1, \ldots, \beta_{r-1}) \mid \beta_i \in \mathcal{N}_p(\mathfrak{g}), [\beta_i, \beta_j] = 0 \} $.\end{center}
\vspace{0.05 in}

By \cite[Lemma 1.7]{SFB1} the variety $V_{G_{(r)}}$ can be naturally identified with $C_r(\mathcal{N}_p(\mathfrak{g}))$ (this is true for any subgroup of $GL_n$ which is of exponential type, and thus in particular for $G$ classical).  That is, every morphism of affine group schemes from $\mathbb{G}_{a(r)}$ to $G_{(r)}$ is given by the data of an $r$-tuple of pairwise commuting $p$-nilpotent elements in $\mathfrak{g}$.  If we write $\underline{\beta} = (\beta_0, \ldots, \beta_{r-1})$, then we follow \cite{SFB1} in writing the corresponding morphism as

\vspace{0.05 in}
\begin{center}$\text{exp}_{\underline{\beta}}: \mathbb{G}_{a(r)} \rightarrow G_{(r)}$,\end{center}
\vspace{0.05 in}

\noindent where, for all commutative $k$-algebras $A$, and all $s \in \mathbb{G}_{a(r)}(A)$,

\vspace{0.05 in}
\begin{center}$\text{exp}_{\underline{\beta}}(s) = \text{exp}_{\beta_0}(s) \cdot \text{exp}_{\beta_1}(s^p) \cdots \text{exp}_{\beta_{r-1}}(s^{p^{r-1}})$.\end{center}
\vspace{0.05 in}

\section{The action of $d\text{exp}_{\underline{\beta}}(u_{r-1})$}

For this section only, we will assume that $G$ is any subgroup of $GL_n$ which is of exponential type.  As just recalled in the previous section, for a $G_{(r)}$-module $M$, the support variety of $M$ over $G_{(r)}$ is given by the set of those $\underline{\beta} = (\beta_0,\ldots,\beta_{r-1}) \in C_r(\mathcal{N}_p(\mathfrak{g}))$ such that $M$ is not projective as a module over $k[u]/(u^p)$, with $u$ acting as $d\text{exp}_{\underline{\beta}}(u_{r-1}) \in Dist(G_{(r)})$.  In this section we prove that this is equivalent to checking whether or not $M$ is projective over $k[u]/(u^p)$ with $u$ acting as $d\text{exp}_{\beta_0}(u_{r-1}) + \cdots + d\text{exp}_{\beta_{r-1}}(u_0)$.

\bigskip
First we state a lemma which recalls two important results pertaining to the projectivity of $k[u]/(u^p)$-modules.  These results can be found in \cite{CLN} and \cite{SFB2}.

\begin{lemma}\label{proj}
Let $M$ be a finite dimensional $k$-vector space.

\begin{enumerate}
\item \cite[Proposition 8]{CLN} Let $x,y \in End_k(M)$ be commuting elements such that $x^p = y^{p-1} = 0$.  Then $M$ is projective as a $k[u]/(u^p)$-module where $u$ acts via $x$ if and only if it is projective when $u$ acts via $x + y$.
\item \cite[Lemma 6.4]{SFB2} Let $x_1,\ldots,x_m \in End_k(M)$ be pairwise commuting $p$-nilpotent elements, and let $f \in k[X_1,\ldots,X_m]$ be a polynomial without linear or constant terms.  Then $M$ is projective as a $k[u]/(u^p)$-module where $u$ acts via $x_1$ if and only if it is projective when $u$ acts via $x_1 + f(x_1,\ldots,x_m)$.
\end{enumerate}
\end{lemma}

\begin{remark}
It should be noted that the actual statement found in \cite[Proposition 8]{CLN} is given in terms of group rings, though it is equivalent to our formulation above.
\end{remark}

\bigskip
Returning now to $d\text{exp}_{\underline{\beta}}$, we note that $\text{exp}_{\underline{\beta}}$ in fact defines a morphism $\mathbb{G}_a \rightarrow G$, which of course restricts to the element in $V_{G_{(r)}}$ detailed earlier.  Evidently all such morphisms from $\mathbb{G}_{a(r)}$ to $G_{(r)}$ arise as restrictions of morphisms from $\mathbb{G}_a \rightarrow G$ (this is observed towards the end of \cite[Remark 1.3]{SFB1} for $GL_n$, and is therefore also true for $G$ a subgroup of exponential type).  It will be convenient to work within this context, viewing $d\text{exp}_{\underline{\beta}}$ as a map from $Dist(\mathbb{G}_a)$ to $Dist(G)$.  We will then write $d\text{exp}_{\underline{\beta}}$ as the composite of two algebra maps, and by keeping track of the image of $u_{r-1}$ under these maps, we will prove Proposition \ref{equal}.

To obtain this decomposition of $d\text{exp}_{\underline{\beta}}$, we see that $\text{exp}_{\underline{\beta}}: \mathbb{G}_a \rightarrow G$ is given by the composite of morphisms $\varphi_r \circ \psi_r$, where

$$\psi_r: \mathbb{G}_a \rightarrow \mathbb{G}_a^{\times r} \quad \text{ and } \quad \varphi_r: \mathbb{G}_a^{\times r} \rightarrow G,$$

\bigskip
\noindent and these morphisms are defined by

$$\psi_r(s) = (s,s^p,\ldots,s^{p^{r-1}}) \text{ and } \varphi_r((a_0, \ldots, a_{r-1})) = \text{exp}_{\beta_0}(a_0) \cdots \text{exp}_{\beta_{r-1}}(a_{r-1}),$$

\bigskip
\noindent for all $s,a_i \in \mathbb{G}_a(A)$, and all $A$.

\bigskip
That $\varphi_r$ is a morphism of affine group schemes (rather than just a morphism of affine schemes) follows from the fact that the $\beta_i$ commute, while $\psi_r$ is fairly clearly a morphism of affine group schemes.  We observe also that, with $F$ and $\delta$ as defined in the previous section,

\begin{equation}\label{recursive}
\psi_2 = (id. \times F) \circ \delta, \quad \text{ and for } i>2, \; \psi_{i} := (id. \times \psi_{i-1}) \circ \psi_2
\end{equation}

\bigskip
\begin{prop}\label{equal}
Let $\underline{\beta} = (\beta_0,\ldots,\beta_{r-1}) \in C_r(\mathcal{N}_p(\mathfrak{g}))$.  Then a $G_{(r)}$-module is projective over $k[u]/(u^p)$ with $u$ acting via $d\textup{exp}_{\underline{\beta}}(u_{r-1}) \in Dist(G_{(r)})$ if and only if it is projective when the action of $u$ is given by

\vspace{0.05 in}
\begin{center} $d\textup{exp}_{\beta_0}(u_{r-1}) + d\textup{exp}_{\beta_1}(u_{r-2}) + \cdots + d\textup{exp}_{\beta_{r-1}}(u_0).$\end{center}
\vspace{0.05 in}

\end{prop}

\begin{proof}
As the morphism $\text{exp}_{\underline{\beta}} = \varphi_r \circ \psi_r$, we have by basic properties (\cite[I.7]{J1}) that $d\text{exp}_{\underline{\beta}} = d\varphi_r \circ d\psi_r$.  From the definition of $\varphi_r$, it follows that

$$d\varphi_r: Dist(\mathbb{G}_a)^{\otimes r} \rightarrow Dist(G)$$

\bigskip
\noindent sends the simple tensor $x_0 \otimes \cdots \otimes x_{r-1}$ to $d\text{exp}_{\beta_0}(x_0)d\text{exp}_{\beta_1}(x_1)\cdots d\text{exp}_{\beta_{r-1}}(x_{r-1})$.  This is true because $\varphi_r$ is $\text{exp}_{\beta_0} \times \cdots \times \text{exp}_{\beta_{r-1}}: \mathbb{G}_a^{\times r} \rightarrow G^{\times r}$, followed by repeated multiplication in $G$, so 

\vspace{0.05 in}
\begin{center} $d\varphi_r = d\text{exp}_{\beta_0} \overline{\otimes} d\text{exp}_{\beta_1} \overline{\otimes} \cdots \overline{\otimes} d\text{exp}_{\beta_{r-1}}.$ \end{center}
\vspace{0.05 in}

To prove the proposition, we will thus show that

\begin{equation}\label{claim}
d\psi_r(u_{r-1}) = u_{r-1} \otimes 1 \otimes \cdots \otimes 1 + 1 \otimes u_{r-2} \otimes 1 \otimes \cdots \otimes 1 + \cdots + 1 \otimes \cdots \otimes 1 \otimes u_0 + y,
\end{equation}

\noindent where $y = y_1 + \cdots + y_m$, and each $y_i$ is the product of two or more $p$-nilpotent elements in $Dist(\mathbb{G}_a)^{\otimes r}$.  Once proved, the result will then follow since this will give us that 

$$d\text{exp}_{\underline{\beta}}(u_{r-1}) = d\text{exp}_{\beta_0}(u_{r-1}) + \cdots + d\text{exp}_{\beta_{r-1}}(u_0) + d\varphi_r(y)$$

\noindent and, according to the second statement in Lemma \ref{proj}, the last term on the right side of the equation will not factor into the detection of projectivity.

Now, to prove claim (\ref{claim}), we see that this holds trivially if $r = 1$, since $\psi_1 = id.$, and hence $d\psi_1 = id.$ as well.  By (\ref{recursive}), and recalling that $d\delta = \Delta_{\mathbb{G}_a}^{\prime}$, we have when $r=2$ that

$$d\psi_2(u_1) = (id. \otimes dF) \circ \Delta_{\mathbb{G}_a}^{\prime}(u_1) = \sum_{i+j = p} \frac{d}{dt}^{(i)} \otimes dF\left(\frac{d}{dt}^{(j)}\right)$$

As $dF(\frac{d}{dt}^{(j)})$ is non-zero only when $j=0$ or $j=p$, we get 

\vspace{0.05 in}
\begin{center}$d\psi_2(u_1) = u_1 \otimes 1 + 1 \otimes u_0$,\end{center}
\vspace{0.05 in}

\noindent thus the claim is also true for $r=2$.  More generally, we have by the second equality in (\ref{recursive}) that $d\psi_r = (id. \otimes d\psi_{r-1}) \circ d\psi_2$.  We first observe then that

$$d\psi_2(u_{r-1}) = (id. \otimes dF) \circ \Delta_{\mathbb{G}_a}^{\prime}(u_{r-1}) = \sum_{i+j = p^{r-1}} \frac{d}{dt}^{(i)} \otimes dF\left(\frac{d}{dt}^{(j)}\right)$$

\bigskip
In $Dist(\mathbb{G}_a) \otimes Dist(\mathbb{G}_a)$, the only values of $i$ for which $\frac{d}{dt}^{(i)} \otimes dF\left(\frac{d}{dt}^{(j)}\right)$ is not the product of two $p$-nilpotent elements are $i = 0$ and $i = p^{r-1}$.  Hence

$$d\psi_2(u_{r-1}) = u_{r-1} \otimes 1 + 1 \otimes u_{r-2} + \sum x_i$$

\bigskip \noindent where each $x_i$ is the product of two $p$-nilpotent elements in $Dist(\mathbb{G}_a)^{\otimes 2}$.  Therefore

$$d\psi_r(u_{r-1}) = (id. \otimes d\psi_{r-1})(u_{r-1} \otimes 1 + 1 \otimes u_{r-2} + \sum x_i)$$

We see then that 

\vspace{0.05 in}
\begin{center}$(id. \otimes d\psi_{r-1})(u_{r-1} \otimes 1) = u_{r-1} \otimes 1 \otimes \cdots \otimes 1$,\end{center}
\vspace{0.05 in}

\noindent while $(id. \otimes d\psi_{r-1})(\sum x_i) = \sum (id. \otimes d\psi_{r-1})(x_i)$ is a sum of elements which are each the product of two $p$-nilpotent elements in $Dist(\mathbb{G}_a)^{\otimes r}$.  Finally, by induction it follows that

$$(id. \otimes d\psi_{r-1})(1 \otimes u_{r-2}) = 1 \otimes \big(u_{r-2} \otimes 1 \otimes \cdots \otimes 1 + \cdots + 1 \otimes \cdots \otimes 1 \otimes u_0  + y^{\prime} \big)$$

\bigskip
\noindent with $y^{\prime} = y^{\prime}_1 + \cdots + y^{\prime}_{m^{\prime}}$, and each $y^{\prime}_i$ is the product of two or more $p$-nilpotent elements in $Dist(\mathbb{G}_a)^{\otimes r-1}$.  This finishes the proof.

\end{proof}

\section{Results for simple modules}

We return now to the case that $G$ is a classical simple group, and look at describing the support varieties of simple $G$-modules over $G_{(r)}$ in terms of support varieties over $G_{(1)}$.  We have chosen to first state this result, in Proposition \ref{first}, for $L(\lambda)$ with $\lambda \in X_1(T)$, and then proceed to the general case in Theorem \ref{simple}, the proof of which will follow from the proof of Proposition \ref{first} together with Steinberg's tensor product theorem on simple $G$-modules.

Our proof relies on a result of  E. Friedlander in \cite{F}, which itself is based on the work of Carlson, Lin, and Nakano in \cite[4.6]{CLN}.  For these methods to be applied then, we need to assume that $p$ is large enough for the root system of the given group.  Specifically, we will assume that $p>hc$, where $h$ is the Coxeter number of $G$, and where $c$ is an integer for the root system of $G$, the value of which is computed in \cite[6.1]{CLN}, and recalled below:

\bigskip
\begin{center}$\begin{array}{lccccc}
\text{Type} & : & A_{\ell} & B_{\ell} & C_{\ell} & D_{\ell}\\
\text{Value of } c & : & (\frac{\ell + 1}{2})^2 & \frac{\ell(\ell + 1)}{2} & \frac{\ell^2}{2} & \frac{\ell(\ell - 1)}{2}\\
\end{array}$ \end{center}

\bigskip
\begin{prop}\label{first}
Assume that $p > hc$.  If $L(\lambda)$ is a simple $G$-module such that $\lambda \in X_1(T)$, then

$$V_{G_{(r)}}(L(\lambda)) = \{ (\beta_0, \ldots, \beta_{r-1}) \in C_r(\mathcal{N}_p(\mathfrak{g}))\mid \beta_{r-1} \in V_{G_{(1)}}(L(\lambda))\}$$
\end{prop}

\begin{proof}
By Theorem \ref{alternate} and Propostion \ref{equal}, we have that $\underline{\beta} = (\beta_0, \ldots, \beta_{r-1}) \in V_{G_{(r)}}(L(\lambda))$ if and only if $L(\lambda)$ is non-projective as a $k[u]/(u^p)$-module with $u$ acting via $d\text{exp}_{\beta_0}(u_{r-1}) + \cdots + d\text{exp}_{\beta_{r-1}}(u_0)$.  To prove this proposition, we will show that, in terms of detection of projectivity, we can ignore the actions of the terms $d\text{exp}_{\beta_{0}}(u_{r-1}) , d\text{exp}_{\beta_{1}}(u_{r-2}), \ldots, d\text{exp}_{\beta_{r-2}}(u_1)$.

For $0 \le i \le r-1$, $\beta_i$ is a $p$-nilpotent element in $\mathfrak{g} \subseteq \mathfrak{gl}_n$, since we are viewing $G$ with its natural embedding in $GL_n$.  As this embedding is of exponential type, $x := exp(\beta_i)$ is a $p$-unipotent element in $G(k) \subseteq GL_n(k)$.  It follows from the discussion in Example 1.12 of \cite{F} that the map $\phi_x : \mathbb{G}_a \rightarrow G$ in Theorem 1.7 of \cite{F} is the same as the map $\text{exp}_{\beta_i}: \mathbb{G}_a \rightarrow G$.  Indeed, it is noted in the example that $\phi_x(t) = exp(t \cdot log(x))$, where $log(x) = \sum_{0<j<p} (-1)^{j+1}(x-1)^j/j$.  One can check then that $log(exp(\beta_i)) = \beta_i$.

We can therefore apply Proposition 2.7 of \cite {F} to $d\text{exp}_{\beta_i}$, which shows that

$$d\text{exp}_{\beta_i}(u_m)^{p-1} = d\text{exp}_{\beta_i}(u_m^{p-1}) = d\text{exp}_{\beta_i}\left((p-1)!\frac{d}{dt}^{(p^m(p-1))}\right)$$

\noindent acts trivially on $L(\lambda)$ if 

$$p^m(p-1) > 2 \sum_{j=1}^{\ell} \langle \lambda, \omega_j^{\vee} \rangle.$$

Following \cite{CLN}, let $(b_{ij})$ denote the inverse matrix of $(\langle \alpha_i, \alpha_j^{\vee} \rangle)$.  From Section 4.6 of \cite{CLN}, we see that 

$$\sum_{j=1}^{\ell} \langle \lambda, \omega_j^{\vee} \rangle = \sum_{i=1}^{\ell} \sum_{j=1}^{\ell} \langle \lambda, \alpha_i^{\vee} \rangle b_{ij}.$$

Tracing through the calculations in \cite[6.1]{CLN}, we further have that $\langle \lambda, \alpha_0^{\vee} \rangle c \ge 2 \sum_{i=1}^{\ell} \sum_{j=1}^{\ell} \langle \lambda, \alpha_i^{\vee} \rangle b_{ij}$, which by the earlier equality says that

$$\langle \lambda, \alpha_0^{\vee} \rangle c \ge 2 \sum_{j=1}^{\ell} \langle \lambda, \omega_j^{\vee} \rangle,$$

\noindent where $c$ is given as above.  Therefore, under the assumption that $p > hc$ and $\lambda \in X_1(T)$, we see that for all $m \ge 1$:

$$p^m(p-1) > hc(p-1) > \langle \rho(p-1), \alpha_0^{\vee} \rangle c \ge \langle \lambda, \alpha_0^{\vee} \rangle c \ge 2 \sum_{j=1}^{\ell} \langle \lambda, \omega_j^{\vee} \rangle.$$

By repeatedly applying the first statement of Lemma \ref{proj}, this then says that $\underline{\beta} \in V_{G_{(r)}}(L(\lambda))$ if and only if $L(\lambda)$ is not projective as a $k[u]/(u^p)$-module with $u$ acting via $d\text{exp}_{\beta_{r-1}}(u_0)$.  Finally, it is not hard to show that $d\text{exp}_{\beta_{r-1}}(u_0) = \beta_{r-1} \in \mathfrak{g}$, which completes the proof.
\end{proof}

\bigskip

\begin{thm}\label{simple}
Assume that $p > hc$.  Let $L(\lambda)$ be a simple $G$-module, with $\lambda$ written $\lambda = \lambda_0 + p\lambda_1 + \cdots + p^{q}\lambda_{q}, \, \lambda_i \in X_1(T)$.  Then

\vspace{0.05 in}
\begin{center}$V_{G_{(r)}}(L(\lambda)) = \{ (\beta_0, \ldots, \beta_{r-1}) \in C_r(\mathcal{N}_p((\mathfrak{g})) \mid \beta_i \in V_{G_{(1)}}(L(\lambda_{r-i-1}))\}$\end{center}
\vspace{0.05 in}

\end{thm}

\bigskip
\begin{proof}
By Steinberg's tensor product theorem,

$$L(\lambda) \cong L(\lambda_0) \otimes L(\lambda_1)^{(1)} \otimes \cdots \otimes L(\lambda_{q})^{(q)}.$$

\bigskip
Now, for $G_{(r)}$-modules $M$,$N$, we have $V_{G_{(r)}}(M \otimes N) = V_{G_{(r)}}(M) \cap V_{G_{(r)}}(N)$ (\cite[7.2]{SFB2}).  Additionally, $M^{(i)}$ is trivial over $G_{(r)}$ if $i \ge r$, so that the statement of the theorem reduces to computing $V_{G_{(r)}}(L(\lambda_i)^{(i)})$, for $i < r$.

We note that the standard Frobenius morphism $F$ acting on $GL_n$ can also be applied to $\mathfrak{gl}_n$ (raising each entry to the $p$-th power).  Because $G$ is a classical subgroup of $GL_n$, $F(\mathfrak{g}) = \mathfrak{g}$.  We have then that $F^i \circ \text{exp}_{\underline{\beta}} = \text{exp}_{F^i(\underline{\beta})} \circ F^i$, where $F^i(\underline{\beta}) = (F^i(\beta_0),F^i(\beta_1),\ldots,F^{i}(\beta_{r-1}))$.

It follows that $L(\lambda_i)^{(i)}$, viewed as a $k[u]/(u^p)$-module with the action of $u$ given by $d\text{exp}_{\beta_0}(u_{r-1}) + \cdots + d\text{exp}_{\beta_{r-1}}(u_0)$, is isomorphic as a $k[u]/(u^p)$-module to the $G$-module $L(\lambda_i)$, with $u$ acting via 

\begin{eqnarray*}
& & d\text{exp}_{F^i(\beta_0)}(dF^i(u_{r-1})) + \cdots + d\text{exp}_{F^i(\beta_{r-1})}(dF^i(u_0)) \\
& = & d\text{exp}_{F^i(\beta_0)}(u_{r-i-1}) + \cdots + d\text{exp}_{F^i(\beta_{r-i-1})}(u_0).
\end{eqnarray*}

\vspace{0.02 in}
However, as shown in the proof of the previous proposition, the assumptions on $p$ and $\lambda_i$ ensure that if $u$ acts as $d\text{exp}_{F^i(\beta_0)}(u_{r-i-1}) + \cdots + d\text{exp}_{F^i(\beta_{r-i-1})}(u_0)$ on $L(\lambda_i)$, then the projectivity of this module over $k[u]/(u^p)$ is determined only by the action of $d\text{exp}_{F^i(\beta_{r-i-1})}(u_0)$.  Thus, $\underline{\beta} \in V_{G_{(r)}}(L(\lambda_i)^{(i)})$ if and only if $F^i(\beta_{r-i-1}) \in V_{G_{(1)}}(L(\lambda_i))$.  Since $V_{G_{(1)}}(L(\lambda_i))$ is stable under the adjoint action of $G$, and because nilpotent orbits in $\mathfrak{g}$ are stable under $F$ (see \cite[\S 3]{S}), it follows that $\underline{\beta} \in V_{G_{(r)}}(L(\lambda_i)^{(i)})$ if and only if $\beta_{r-i-1} \in V_{G_{(1)}}(L(\lambda_i))$.  Taking the intersection over all $0 \le i \le r-1$ finishes the proof.

\end{proof}

\begin{remark}
It is a basic property of support varieties that for any finite dimensional $G_{(r)}$-module $M$, $\text{dim} \, V_{G_{(r)}}(M) = c_{G_{(r)}}(M)$, where $ c_{G_{(r)}}(M)$ is the complexity of $M$.  In work which predates \cite{SFB1} and \cite{SFB2}, D. Nakano was able to prove \cite[2.4]{N} that 

$$ c_{G_{(r)}}(L(\lambda)) \le \sum_{i=0}^{r-1} c_{G_{(1)}}(L(\lambda_i)),$$

\noindent a bound which is now also verified by Theorem \ref{simple}.  It is worth noting however that Nakano's proof is independent of the prime $p$, and applies to any $G$ which is connected and semisimple.  There is thus some hope that our results will hold in much more generality. 
\end{remark}

\section{Support Varieties of Blocks}

In this section we will assume further that $G$ is simply-connected, thus $G$ is either a special linear group, or a symplectic group.  The finite dimensional algebra $Dist(G_{(r)})$ has a decomposition into blocks, and a $G_{(r)}$-module $M$ is said to lie in a block $\mathcal{B}$ if the central idempotent of $\mathcal{B}$ acts as the identity on $M$.  By the support variety of $\mathcal{B}$, we mean the union of the support varieties of all modules lying in $\mathcal{B}$.  It follows from general properties of support varieties that the support variety of $\mathcal{B}$ is equal to the union of the support varieties of all simple $G_{(r)}$-modules lying in $\mathcal{B}$.

The goal of this section is to use the results of the previous section, along with results from \cite{J2} and \cite{NPV}, to give an explicit description of these varieties.  For $\lambda \in X(T)_+$, let $\mathcal{B}_r(\lambda)$ denote the block of $Dist(G_{(r)})$ which contains the simple $G$-module $L(\lambda)$.  We observe that if $\lambda = \lambda_1 + p^r\lambda_2$ with $\lambda_1 \in X_r(T)$ and $\lambda_2 \in X(T)_+$, then as a $G_{(r)}$-module $L(\lambda)$ is isomorphic to the direct sum of $dim_k(L(\lambda_2))$ copies of $L(\lambda_1)$, the latter remaining simple upon restriction to $G_{(r)}$, thus $L(\lambda)$ does indeed lie in a single block.  We write $\mu \in \mathcal{B}_r(\lambda)$ if $L(\mu)$ lies in $\mathcal{B}_r(\lambda)$, so that $\mathcal{B}_r(\lambda)$ defines a subset of $X(T)_+$.  It is clear that $\mathcal{B}_r(\lambda) = \mathcal{B}_r(\mu)$ whenever $\mu \in \mathcal{B}_r(\lambda)$, and since every block has some simple $G_{(r)}$-module lying in it, then every block of $Dist(G_{(r)})$ can be given as $\mathcal{B}_r(\lambda)$ with $\lambda \in X_r(T)$.

Jantzen has shown (\cite[II.9.22]{J1}) how to determine the block $\mathcal{B}_r(\lambda)$ (and this result is true for arbitrary reductive groups): let $m$ be the smallest integer such that there is some $\alpha \in \Phi$ with $\langle \lambda + \rho, \alpha^{\vee} \rangle \notin \mathbb{Z}p^m$.  Then the block $\mathcal{B}_r(\lambda)$ contains $L(\mu)$ if and only if

\begin{equation}\label{link}
\mu \in \big(W \cdot \lambda + p^m\mathbb{Z}\Phi + p^rX(T)\big) \cap X(T)_+ .
\end{equation}

\bigskip
In applying this result, we note that for the groups under consideration in this section, the index of $\mathbb{Z}\Phi$ in $X(T)$ is $n$ if $G = SL_n$, and is $2$ if $G = Sp_{2n}$ (\cite[13.1]{Hum}).  Thus if we are assuming that $p>hc$, we have then $bX(T) \subseteq \mathbb{Z}\Phi$ for some integer $b$ which is relatively prime to $p$.  Therefore, if $m \le r$, we have

$$p^mX(T) \supseteq p^m\mathbb{Z}\Phi + p^rX(T) \supseteq  p^m(bX(T) + p^{r-m}X(T)) = p^mX(T),$$

\bigskip
\noindent so we can simplify (\ref{link}) to

\begin{equation}\label{newlink}
\mu \in \big(W \cdot \lambda + p^mX(T)\big) \cap X(T)_+ .
\end{equation}

\bigskip
Let us now recall how the work of Jantzen in \cite{J2}, and Nakano, Parshall, and Vella in \cite{NPV}, provide for the first Frobenius kernel of $G$ the description of the variety of the block $\mathcal{B}_1(\lambda)$.  As with the result for determining blocks, this holds for arbitrary reductive groups, but we will continue with our assumption that $G$ is $SL_n$ or $Sp_{2n}$.

Following the notation in \cite{NPV}, let $\Phi_{\lambda} := \{ \alpha \in \Phi \mid \langle \lambda + \rho, \alpha^{\vee} \rangle \in p\mathbb{Z} \}$.  Assuming that $p$ is good for $\Phi$ (which in particular is true if $p>hc$), it is observed in \cite[6.2]{NPV} that there exists $w \in W$ and a subset $I \subseteq \Pi$ such that $w(\Phi_{\lambda}) = \mathbb{Z}I \cap \Phi$.  Define $\mathfrak{u}_I \subseteq \mathfrak{g}$ to be the subalgebra generated by all $\text{Lie}(U_{-\alpha})$, where $\alpha$ is a positive root not contained in $\mathbb{Z}I$.  As proved in \cite{J2} for type $A$, and in \cite{NPV} for all other types in which $p$ is good:

\begin{equation}\label{induced}
V_{G_{(1)}}(\text{H}^0(\lambda)) = G \cdot \mathfrak{u}_I
\end{equation}

\bigskip
\noindent where $G$ acts on $\mathfrak{g}$ by the adjoint action.  Note that if $\mu \in \mathcal{B}_1(\lambda)$, then $\mu = w^{\prime} \cdot \lambda + p\beta$, for some $\beta \in X(T)$, so that $\Phi_{\mu} = w^{\prime}(\Phi_{\lambda})$.  Thus $w(w^{\prime})^{-1}(\Phi_{\mu}) = \mathbb{Z}I \cap \Phi$, and by (\ref{induced}) we have $V_{G_{(1)}}(\text{H}^0(\lambda)) = V_{G_{(1)}}(\text{H}^0(\mu))$.  Applying Theorem 4.6.1 of \cite{NPV}, and its proof, it follows that the support variety of $\mathcal{B}_1(\lambda)$ equals $G \cdot \mathfrak{u}_I$.

\bigskip
We can now proceed to state and prove the result for higher Frobenius kernels.

\begin{thm}\label{block}
Let $\lambda \in X_r(T)$, and let $\mathcal{B}_r(\lambda)$ denote the block of $Dist(G_{(r)})$ containing the simple $G$-module $L(\lambda)$, and assume that $p>hc$.  Write

\vspace{0.05 in}
\begin{center}$\lambda = \lambda_0 + \lambda_1p + \cdots + \lambda_{r-1} p^{r-1} , \, \lambda_i \in X_1(T)$\end{center}
\vspace{0.05 in}

\noindent and let $m>0$ be the smallest integer such that $\lambda_{m-1} \ne (p-1)\rho$.  Then the support variety of $\mathcal{B}_r(\lambda)$ is given by

\vspace{0.05 in}
\begin{center}$ \{ (\beta_0, \ldots, \beta_{r-1}) \in C_r(\mathcal{N}_p(\mathfrak{g})) \mid \beta_{r-m} \in V_{G_{(1)}}(\textup{H}^0(\lambda_{m-1})) \text{ and } \beta_i =0 \text{ if } i > r-m \}$.\end{center}
 
\end{thm}

\vspace{0.05 in}
\begin{proof}
We can write $\lambda = \rho(p^{m-1} -1) + \lambda_{m-1}p^{m-1} + \sigma p^m$, where 

$$\sigma = \lambda_m + \lambda_{m+1}p + \cdots + \lambda_{r-1}p^{r-m-1}$$

For any $w \in W$, we have

\begin{eqnarray*}
w \cdot (\lambda) & = & w\big(\rho(p^{m-1} - 1) + \lambda_{m-1}p^{m-1} + \sigma p^m + \rho \big) - \rho \\
& = &  w(\rho + \lambda_{m-1})p^{m-1} + w(\sigma) p^m - \rho \\
& = & \big(w(\rho + \lambda_{m-1}) - \rho \big)p^{m-1} + \rho(p^{m-1} - 1) + w(\sigma) p^m\\
& = & \rho(p^{m-1} - 1) + (w \cdot \lambda_{m-1}) p^{m-1} + w(\sigma) p^m\\
\end{eqnarray*}

By (\ref{newlink}), $L(\mu)$ is in the block $\mathcal{B}_r(\lambda)$ if and only if $\mu$ is both dominant and in $W \cdot \lambda + p^mX(T)$.  Write $\mu = \mu_0 + \mu_1p + \cdots + \mu_q p^q, \, \mu_i \in X_1(T)$.  Having just calculated $w \cdot \lambda$, it follows that if $\mu \in \mathcal{B}_r(\lambda)$, then $\mu_i = \rho(p -1)$, if $0 \le i < m-1$, and $\mu_{m-1} \in \big(W \cdot \lambda_{m-1} + pX(T) \big) \cap X(T)_+ = \mathcal{B}_1(\lambda_{m-1})  \cap X(T)_+$.  We further note that by subtracting $\mu_mp^m + \cdots + \mu_q p^q \in p^mX(T)$ from $\mu$, we have that $L(\mu^{\prime})$ is also in $\mathcal{B}_r(\lambda)$, where $\mu^{\prime} = \mu_0 + \mu_1p + \cdots + \mu_{m-1} p^{m-1}$.  By Theorem \ref{simple} it is clear that $V_{G_{(r)}}(L(\mu)) \subseteq V_{G_{(r)}}(L(\mu^{\prime}))$.  Thus we see that the support variety of the block $\mathcal{B}_r(\lambda)$ is given by the union of the support varieties of the modules $L(\mu)$ with $\mu \in \big(W \cdot \lambda + p^mX(T)\big) \cap X_m(T)$.  By the observations above, 

$$\big(W \cdot \lambda + p^mX(T)\big) \cap X_m(T) = \{ \rho(p^{m-1}-1) + \sigma p^{m-1} \mid \sigma \in \mathcal{B}_1(\lambda_{m-1}) \cap X_1(T) \}.$$

\bigskip
Let $\mathcal{B}^{\prime}_1(\lambda_{m-1}) = \mathcal{B}_1(\lambda_{m-1}) \cap X_1(T)$.  Summarizing, we have that the support variety of the block $\mathcal{B}_r(\lambda)$ is equal to

\begin{eqnarray*}
& = & \bigcup_{\sigma \in \mathcal{B}^{\prime}_1(\lambda_{m-1})} V_{G_{(r)}}\left( L(\rho(p^{m-1}-1)) \otimes L(\sigma)^{(m-1)} \right)\\
\, &  & \\
& = & V_{G_{(r)}}\left(L(\rho(p^{m-1}-1))\right) \,\; \bigcap \,\; \left( \bigcup_{\sigma \in \mathcal{B}^{\prime}_1(\lambda_{m-1})} V_{G_{(r)}}\left(L(\sigma)^{(m-1)}\right) \right)\\
\end{eqnarray*}

\bigskip
If $m=1$, then $L(\rho(p^{m-1}-1)) = L(0)$, which is the trivial $G_{(r)}$-module, otherwise if $m \ge 2$, 

\vspace{0.05 in}
\begin{center}$L(\rho(p^{m-1}-1)) \cong L(\rho(p-1)) \otimes L(\rho(p-1))^{(1)} \otimes \cdots \otimes L(\rho(p-1))^{(m-2)}.$\end{center}
\vspace{0.05 in}

Since $L(\rho(p-1))$ is projective over $G_{(1)}$, we have $V_{G{(1)}}(L(\rho(p-1))) = \{ 0 \}$.  Thus, for all $m \ge 1$ we have by Theorem \ref{simple},

\vspace{0.05 in}
\begin{center}$V_{G_{(r)}}(L(\rho(p^{m-1}-1))) = \{ (\beta_0, \ldots, \beta_{r-m},0,\ldots,0) \in C_r(\mathcal{N}_p(\mathfrak{g})) \}.$\end{center}
\vspace{0.05 in}

The calculation of the support variety of $ \mathcal{B}_1(\lambda_{m-1})$ over $G_{(1)}$ together with Theorem \ref{simple} implies that

$$\bigcup_{\sigma \in \mathcal{B}^{\prime}_1(\lambda_{m-1})} V_{G_{(r)}}(L(\sigma)^{(m-1)}) = \{ \underline{\beta} \in C_r(\mathcal{N}_p(\mathfrak{g})) \mid \beta_{r-m} \in \Phi_{\lambda_{m-1}} \},$$

\noindent which proves the theorem.
\end{proof}

\begin{remark}
As observed in the comments at the beginning of Section 4.6 of \cite{NPV}, the module $\text{H}^0(\lambda)$ lies in the block $\mathcal{B}_r(\lambda)$, thus this theorem provides an upper bound on $V_{G_{(r)}}(\text{H}^0(\lambda))$.  In the case that $G = SL_2$ and $r \ge 1$, it follows from the calculations in \cite[6.10]{SFB2} that the support variety of $\text{H}^0(\lambda)$ is in fact equal to the support variety of $\mathcal{B}_r(\lambda)$.  The same is true in the case $r=1$ and $G$ a reductive group, as follows from the work in \cite{NPV}.  Possibly this equality holds for arbitrary $r$ and $G$ reductive (or at least a classical simple group), though we will stop short of officially stating this as a conjecture.
\end{remark}

\subsection{Acknowledgments} This paper is the continuation of work which appeared in the author's Ph.D. thesis, and we gratefully acknowledge the influence and help provided by Eric Friedlander, who served as our thesis advisor.  We also thank Julia Pevtsova, as well as the reviewer, for many helpful suggestions and comments, and we acknowledge useful conversations with Jim Humphreys.  This research was partially supported by grants from the Australian Research Council (DP1095831, DP0986774 and DP120101942).

\vspace{.2 in}
\noindent\tiny{DEPARTMENT OF MATHEMATICS \& STATISTICS, UNIVERSITY OF MELBOURNE, PARKVILLE, VIC 3010, AUSTRALIA}\\
paul.sobaje@unimelb.edu.au\\
Phone: \text{+}61 \, 401769982

\end{document}